\documentclass[11pt]{article}

\usepackage[english]{babel}
\usepackage[margin=1in]{geometry}
\usepackage{amsmath,amsthm,amssymb,mathtools}
\usepackage{algorithm,algorithmic}
\usepackage{float,graphicx,epsfig}
\usepackage[labelfont=bf]{caption,subcaption}
\usepackage{enumitem,multirow,multicol}
\usepackage{lipsum} 

\usepackage{microtype}
\usepackage{titlesec}
\usepackage{setspace}
\expandafter\def\expandafter\normalsize\expandafter{%
  \normalsize  
  \setlength\abovedisplayskip{1.83ex}
  \setlength\belowdisplayskip{1.83ex}
  \setlength\abovedisplayshortskip{1.83ex}
  \setlength\belowdisplayshortskip{1.83ex}
}

\usepackage[dvipsnames]{xcolor}
	\definecolor{bondiblue}{rgb}{0.1, 0.183, 0.683}
			\definecolor{ferngreen}{rgb}{0.31, 0.47, 0.26}
\usepackage[colorlinks=true,breaklinks=true,bookmarks=true,urlcolor=bondiblue,linkcolor=Blue,citecolor=Blue,bookmarksopen=false,draft=false]{hyperref}

\makeatletter
\newcommand{\blackstar}{\textcolor{black}{*}}
\renewcommand{\@fnsymbol}[1]{\@arabic\c@footnote}
\def\@fnsymbol#1{\ensuremath{\ifcase#1\or \blackstar\or \dagger\or \ddagger\or
   \mathsection\or \mathparagraph\or \|\or **\or \dagger\dagger
   \or \ddagger\ddagger \else\@ctrerr\fi}}
\makeatother

\titleformat*{\section}{\Large\bf}
\titlespacing*{\section}{0pt}{3ex}{2ex}

\theoremstyle{plain}
\newtheorem{theorem}{Theorem}
\newtheorem{lemma}[theorem]{Lemma}

\theoremstyle{definition}

\newtheorem{example}{Example}
\newtheorem{remark}{Remark}

\numberwithin{equation}{section}

\usepackage{lineno}
\usepackage{soul}

\makeatletter

\renewcommand{\maketitle}{
\begin{center}
\LARGE{\@title} 
\vspace{0.138in}

\large{\@author}
\vspace{0.138in}

\large{\@date}
\end{center}
}

\title{\LARGE{\bf A New Lagrangian-Based First-Order Method for Nonconvex Constrained Optimization}}	
\author{%
    {\large Jong Gwang Kim\footnote{School of Industrial Engineering, Purdue University, West Lafayette, IN 47906; {\href{mailto:kim2133@purdue.edu} {\color{black}{\texttt{kim2133@purdue.edu}.}}}
    }}
}
\date{}

\begin{document}
   
\maketitle

\begin{abstract}
We introduce a new form of Lagrangian and propose a simple first-order algorithm for nonconvex optimization with nonlinear equality constraints. We show the algorithm generates bounded dual iterates, and establish the convergence to KKT points under standard assumptions. The key features of the method are: (i) it does not require boundedness assumptions on the iterates and the set of multipliers; (ii) it is a single-loop algorithm that does not involve any penalty subproblems.
\end{abstract}

\section{Introduction} \label{sec:intro}

Consider the nonconvex optimization problem with nonlinear equality constraints:
\begin{equation} \label{eq:op}
 	\underset{x \in \mathbb{R}^n}{\text{min}} \  f(x) \ \ \ \text{s.\:t.} \ \ \ c(x) = 0, \ \ \ x \in X,
\end{equation}
where $f: \mathbb{R}^n \rightarrow \mathbb{R}$,  $c:=(c_1,\ldots,c_m): \mathbb{R}^n \rightarrow \mathbb{R}^m $, and $f$ and $c_j$, $j=1,\ldots,m$, are continuously differentiable and possibly nonconvex. $X \subseteq \mathbb{R}^n$  is a nonempty closed and convex set. We make the following assumptions:

\begin{enumerate}[label=A\arabic*., ref=(A\arabic*),align=left,leftmargin=.383in,widest={0}]
\setlength\itemsep{-0.038in}
\vspace{-0.0838in}
 	
 \item \label{assumption_lipschitz_i} The gradient  $\nabla f$ is $L_{\nabla f}$-Lipschitz continuous over $X$.
 		
 \item \label{assumption_lipschitz_ii} The Jacobian $\nabla c$ is $L_{\nabla c}$-Lipschitz continuous over $X$.

  \item \label{assumption_lipschitz_iii} The mapping $c$ is $L_c$-Lipschitz continuous over $X$.

 \item \label{assumption_coercive} $X$ is bounded (and thus compact), or, otherwise, $f$ is coercive and lower bounded over $X$.
 \end{enumerate}	
The above Assumptions are quite standard in nonconvex settings; see e.g., \cite{boct2016inertial,li2015global,ochs2014ipiano,yang2017alternating}. Assumption \ref{assumption_coercive} is needed to guarantee that the primal sequence generated by our method is bounded. 
The main objective of this paper is to present a new primal-dual framework for solving nonconvex problem \eqref{eq:op} and provide its theoretical guarantees. We review the literature relevant to this work and relate theoretical results to our contributions.

One of the most popular approaches for solving constrained optimization problems is the augmented Lagrangian (AL) method, in particular within the alternating direction method of multipliers (ADMM) scheme. While the AL-based methods for constrained convex problems have been extensively studied in the literature  (see books \cite{bertsekas2014constrained,birgin2014practical} and recent works \cite{liu2019nonergodic,ouyang2015accelerated,shefi2014rate,xu2017accelerated}), the convergence analysis for AL-based methods applied to constrained nonconvex settings remains fairly limited due to the challenges caused by nonconvexity of the objective and constraint functions. Most AL-based methods thus established the convergence with some restrictive assumptions.

The convergence results in \cite{conn1996convergence,conn1991globally,friedlander2005globally,lewis2002globally} were established under the linear independence constraint qualification (LICQ) (i.e., the gradients of constraints at limit points are linearly independent). It is well-known that the LICQ and the Mangasarian-Fromovitz constraint qualification (MFCQ) are the same in the context of nonlinear equality-constrained problems, which is equivalent to the set of multipliers being bounded;\cite[Section 1.2.4]{izmailov2014newton} and \cite{solodov2010constraint}.  However, LICQ may be too restrictive since many optimization problems may not have a strict relative interior and thus have an unbounded set of the multipliers \cite{haeser2019behavior}. 
Furthermore, even if the LICQ is assumed, AL-based methods may fail to converge to a stationary solution due to the unbounded iterates of multipliers. Thus, as pointed out in \cite{rockafellar1974augmented}, the AL framework may require a strong assumption that dual iterates are bounded when solving nonlinear programs. As a remedy, the safeguarding technique, which imposes artificial bounds on the dual iterates, has been proposed in \cite{andreani2007augmented,andreani2008augmented}. 

Recently, several AL-based algorithms have been proposed to solve constrained nonconvex problems without requiring boundedness assumptions on the dual iterates; see e.g., \cite{boct2020proximal,hong2016convergence,li2015global,wang2019global,yang2017alternating}. However, their analysis cannot be directly extended to general problem \eqref{eq:op} as they focus on nonconvex problems with linear constraints. The work \cite{bolte2018nonconvex} considered a general class of nonconvex-nonsmooth optimization and provided a general AL framework with global convergence. However, their convergence analysis also relies on the boundedness assumption on the dual iterates.

Motivated by the theoretical limitations of existing AL-based methods, we introduce a new Lagrangian-based algorithm for which convergence to a stationary solution can be established under standard assumptions. 
Assuming a suitable constraint qualification (CQ) holds, the stationary solutions of problem \eqref{eq:op} can be characterized by the points with corresponding multipliers $(x^{\ast}, \lambda^{\ast})$ satisfying the Karush-Kuhn-Tucker (KKT) conditions \cite{andreani2022scaled,bertsekas1999nonlinear}. Throughout the paper, we assume that under a suitable CQ, problem \eqref{eq:op} has at least one point $(x^{\ast}, \lambda^{\ast})$ satisfying the KKT conditions:
\begin{equation} \label{eq:def_FOC}
   0 \in \nabla f(x^\ast) +\nabla c(x^\ast) \lambda^\ast + \mathcal{N}_X( x^\ast), \quad c(x^\ast)=0,
\end{equation}
where $\mathcal{N}_X(x^\ast) = \left\{v \in X \left|\right. \left\langle v, x -x^\ast \right\rangle \leq 0, \forall x \in X \right\}$ is the normal cone to $X$ at $x^\ast$. Note that we allow  weaker CQs (quasinormality CQ, CPLD, and others; see \cite{andreani2022best}) than LICQ  for the existence of multipliers. 

In problem setting \eqref{eq:op}, the set of multipliers may be unbounded in general, even if they satisfy the KKT conditions \eqref{eq:def_FOC}. This makes the computation of a KKT point difficult, so the boundedness of multipliers is one of the key issues when solving nonconvex optimization problems with nonlinear constraints in the AL framework. This challenge motivates us to introduce a new form of a Lagrangian function.
\vspace{0.083in}

\noindent \textbf{Our Contributions.} This paper makes the following contributions to the literature. 
We introduce a new Lagrangian function that has a favorable structure; it is strongly concave with respect to the multipliers and it does not include penalty terms for handling nonlinear constraints. This structure allows us to design a simple single-loop first-order algorithm that produces bounded primal-dual iterates. We prove the algorithm converges to the KKT points, without making boundedness assumptions on dual variables and the LICQ assumption, provided that the Lagrange multipliers exist. 
\vspace{0.083in}

\noindent \textbf{Organization.} 
The rest of the paper is organized as follows. Section \ref{sec:plagrangian} introduces a new Lagrangian function and describes its characteristics. In section \ref{sec:algorithm}, we present a simple primal-dual algorithm, based on the new Lagrangian with proximal linearized approximation. We provide the convergence results of the proposed algorithm in Section \ref{sec:convergence}. Preliminary numerical results are presented in Section \ref{sec:experiments}.
\vspace{0.083in}

\noindent \textbf{Notation.} The space $\mathbb{R}^n$  denotes an $n$-dimensional Euclidean space with inner product $\left\langle x, y \right\rangle$ for $x,y \in \mathbb{R}^n$. The  Euclidean norm of a vector is denoted by $\left\| \: \cdot \: \right\|$ and the Euclidean norm of a matrix is also denoted by $\left\| \: \cdot \: \right\|$. Let $\mathbb{N}=\{0,1,2,\ldots\}$ be the set of nonnegative integers. For a closed convex set $X \subseteq \mathbb{R}^n$, we denote by $\textrm{P}_X$ the projection operator onto $X$, i.e., $\textrm{P}_X \left[ x \right]=\text{arg\,min}_{y \in X} \| x-y \|$.

\section{Proximal-Perturbed Lagrangian Formulation} \label{sec:plagrangian}

Inspired by a reformulation technique in \cite[Chapter 3.4]{bertsekas1989parallel}, we begin by reformulating problem \eqref{eq:op} as an extended formulation by introducing \emph{perturbation} variables $z \in \mathbb{R}^m$ and letting $z = 0$ and $c(x)=z$:
\begin{equation} \label{eq:ep}
 	\underset{x \in X,\: z \in \mathbb{R}^m}
 	{\text{min}} f(x) \quad
 	\text{s. t.} \quad c(x)=z, \quad z=0. \notag
\end{equation} 
Obviously, for the unique solution $z=0$ the above formulation is equivalent to problem \eqref{eq:op}. Let us now define the \emph{Proximal-Perturbed Lagrangian} (P-Lagrangian): 
\begin{equation} \label{eq:modified_AL}
\begin{aligned}    
 \mathcal{L}_\beta(x,z,\lambda,\mu)  = & 
 f(x) + \left\langle \lambda, c(x)-z \right\rangle  
 + \left\langle \mu, z \right\rangle  
  +\frac{\alpha}{2} \| z \|^2  
 -\frac{\beta}{2}\| \lambda - \mu \|^2, \notag
  \end{aligned}
\end{equation}
where $\lambda \in \mathbb{R}^m$ and $\mu \in \mathbb{R}^m$ are the Lagrange multipliers associated with the constraints $c(x)-z=0$ and $z=0$, respectively. $\alpha>0$ is a penalty parameter and $\beta>0$ is a proximal parameter.  

The structure of P-Lagrangian differs from the standard AL function; (1) it is characterized by the absence of penalty term for handling $c(x)-z=0$, and (2) it is strongly concave w.r.t the multiplier $\lambda$ (for fixed $\mu$) and in $\mu$ (for fixed $\lambda$) due to the \emph{dual proximal regularization} term $-\frac{\beta}{2}\| \lambda -\mu \|^2$. 
We observe some properties of $\mathcal{L}_\beta(x,z,\lambda,\mu)$. Given $(\lambda,\mu)$ if we minimize $L_\beta(x,z,\lambda,\mu)$ in $z$, we get a unique solution $\hat{z}(\lambda, \mu) =(\lambda - \mu)/ \alpha$. Substituting $\hat{z}(\lambda, \mu)$ into $\mathcal{L}_\beta(x,z,\lambda,\mu)$, $\mathcal{L}_\beta$ reduces to 
\begin{equation}\label{eq:reducedAL}
   \mathcal{L}_{\beta}(x,\hat{z}(\lambda, \mu),\lambda,\mu) = f(x) + \left\langle \lambda, c(x) \right\rangle - \frac{1}{2\rho} \| \lambda- \mu \|^2.
\end{equation}
Next, since $\mathcal{L}_\beta(x,\hat{z}(\lambda, \mu),\lambda,\mu)$ is strongly concave in $\lambda$, there exists a unique maximizer $\widehat{\lambda}(x,\mu)$. That is, if we maximize the reduced P-Lagrangian \eqref{eq:reducedAL} in $\lambda$, we obtain
\begin{equation} \label{eq:property_lambda}
   \hat{\lambda}(x,\mu) = \underset{\lambda \in \mathbb{R}^m}{\textrm{argmax}} \; \mathcal{L}_\beta(x,\hat{z}(\lambda,\mu),\lambda,\mu) = \mu + \rho c(x). \notag 
\end{equation}

\section{Algorithm} \label{sec:algorithm}

In this section, we present a simple single-loop algorithm, based on the P-Lagrangian features. The steps of the algorithm are summarized in Algorithm \ref{algorithm1}.
 
 \begin{algorithm}[H]	
\caption{P-Lagrangian Alternating Direction Algorithm} \label{algorithm1}
\begin{algorithmic}[1]

\STATE \textbf{Input:} $\alpha \gg 1$, $\beta \in (0,1)$,  $\rho:=\frac{\alpha}{1+\alpha \beta}$, $\eta > L_p +2\rho L_c^2$, $r \in (0,1)$.
	
\STATE \textbf{Initialization:} {Set $(x_0,z_0,\lambda_0,\mu_0)$ and $\delta_0 \in \left(0,1 \right]$.}		
\FOR{$k =0,1,2,\ldots$} 

\STATE\label{step1}  
$x_{k+1} = \mathrm{P}_X \left[ x_k - \eta^{-1} \nabla_x \mathcal{L}_\beta(x_k,z_k,\lambda_k,\mu_k)\right]$.

\STATE\label{step2} 
$\mu_{k+1} = \mu_k + \frac{\gamma_k} {\rho}(\lambda_k-\mu_k)$
with $\gamma_k=\frac{\rho \delta_k}{\| \lambda_k - \mu_k \|^2 + 1}$.

\STATE \label{step3} 
$\lambda_{k+1} = \mu_{k+1} + \rho c(x_{k+1})$.

\STATE\label{step4} 
$z_{k+1} = \frac{1}{\alpha}(\lambda_{k+1} - \mu_{k+1})$.
		
\STATE\label{beta_update} $\delta_{k+1} = r \delta_k$.	
	
\ENDFOR 
\end{algorithmic} 
\end{algorithm}  
Note that exact minimization of $\mathcal{L}_\beta$ in $x$ is difficult in general due to the nonconvexity of $f$ and $c_j$, $j=1,\ldots,m$. To overcome this difficulty, we adopt a simple approximation $\widehat{\mathcal{L}}_\beta$ in only ${x}$ at a given point $y$ (see e.g., \cite{bolte2014proximal}):
\begin{equation} \label{eq:approx_AL}
\begin{aligned}
  \widehat{\mathcal{L}}_\beta (x,z,\lambda,\mu;y) := \mathcal{L}_\beta({y},z,\lambda,\mu)
   +\left\langle \nabla _x \mathcal{L}_\beta({y},z,\lambda,\mu), x - y \right\rangle + \frac{\eta}{2}\| x - y \|^2,
   \end{aligned}
\end{equation}
which is the so-called proximal linearized approximation of $\mathcal{L}_\beta$ in $x$. The algorithm first updates the primal variables $x$ by minimizing  $\widehat{\mathcal{L}}_\beta$ \eqref{eq:approx_AL} in $x$ while fixing  $(z_k,\lambda_k,\mu_k)$:
\begin{equation}\label{eq:x_update}
x_{k+1} 
= \underset{x\in X}{\mathrm{argmin}} \left\lbrace 
\left\langle \nabla_{x}\mathcal{L}_\beta(x_k), x - x_k \right\rangle + \frac{\eta}{2}\| x-x_k \|^2 \right\rbrace,
\end{equation}
which is equivalent to the projected gradient descent;
$x_{k+1} = \mathrm{P}_X \left[ x_k - \eta^{-1} \nabla_x \mathcal{L}_\beta(x_k,z_k,\lambda_k,\mu_k)\right]$,  where a large enough $\eta$ is chosen for convergence of Algorithm \ref{algorithm1} (see Lemma \ref{thm_sufficient_decrease}).

Then the algorithm performs a gradient ascent step for updating the auxiliary multiplier $\mu$:
\begin{equation}\label{eq:mu_update}
   \mu_{k+1} = \mu_k + \gamma_k \nabla_{\mu}\mathcal{L}_\beta(x_{k+1},z_k,\mu_k,\lambda_k)  
   = \mu_k + \left(\frac{\gamma_k}{\rho} \right) (\lambda_k - \mu_k),
\end{equation}
with $\gamma_k$ defined by
\begin{equation} \label{eq:gamma_update}
\gamma_k := \frac{\rho \delta_k}{\| \lambda_k - \mu_k \|^2 + 1}, \textrm{ where } \rho := \frac{\alpha}{1 + \alpha\beta} \  \text{is fixed,} 
\end{equation} 
and $\delta_k>0$ is summable, namely $\sum_{k=0}^\infty \delta_k < \infty$. We use $\delta_k = r^k \delta_0$ with $\delta_0 \in \left(0, 1 \right]$ and a reduction ratio of  $r \in \left(0, 1\right)$. The use of $\gamma_k$ guarantees the boundedness of  $\{\mu_k\}_{k \in \mathbb{N}}$ (Lemma \ref{lem_boundness_art_dual}).

\begin{lemma} \label{lem_boundness_art_dual}
Let $\{(x_k, z_k, \lambda_k, \mu_k) \}_{k \in \mathbb{N}}$ be the sequence generated by Algorithm \ref{algorithm1}. Then,  $\{ \mu_k \}_{k \in \mathbb{N}}$ is bounded. 
\end{lemma}

\begin{proof}
It follows from  \eqref{eq:mu_update} and  \eqref{eq:gamma_update} that
\begin{align}
\| \mu_{k+1} \| 
&  \leq \| \mu_0 \| + \sum^k_{s=0} \frac{\gamma_s}{\rho_s} \| \lambda_s - \mu_s \| \notag \\
& = \| \mu_0 \| + \sum^{\infty}_{s=0} \frac{1}{\rho} \cdot \frac{\rho \delta_s }{\| \lambda_s - \mu_s \| + \frac{1}{\| \lambda_s - \mu_s \|}}  \notag \\
& \leq \| \mu_0 \| + \frac{1}{2}\sum^{\infty}_{s=0} \delta_s, \notag
\end{align}
where in the second inequality, we used $a + b \geq 2\sqrt{ab}$ for $a,b \geq 0$. Notice that $\sum^{\infty}_{s=0} {\delta_s}$ is convergent as $\delta_s = r^s \delta_0$ and  $r \in (0, 1)$. Hence, $\{ \mu_k \}_{k \in \mathbb{N}}$ is bounded.
\end{proof}

Then, the multiplier $\lambda$ is updated by an exact maximization step on $\mathcal{L}_\beta(x,\hat{z}(\lambda,\mu),\lambda,\mu)$ in \eqref{eq:reducedAL}:
\begin{align}    
 \lambda_{k+1}  & = \underset{\lambda \in \mathbb{R}^{m}}{\mathrm{arg\,max}}\left\lbrace f(x_{k+1}) + \left\langle \lambda, c(x_{k+1}) \right\rangle - \frac{1}{2\rho}\| \lambda- \mu_{k+1} \|^2 \right\rbrace \notag \\ & = \mu_{k+1} + \rho  c(x_{k+1}).   \label{eq:lambda_update} 
\end{align}

The final step is to update $z$ via an exact minimization step on $\mathcal{L}_\beta$ with a large value of $\alpha>0$: 
\begin{equation}\label{eq:z_update}
z_{k+1} = \underset{z\in\mathbb{R}^{m}}{\mathrm{argmin}}\left\lbrace \mathcal{L}_\beta(x_{k+1},z,\lambda_{k+1},\mu_{k+1})\right\rbrace
=\frac{(\lambda_{k+1}-\mu_{k+1})}{\alpha}. 
\end{equation}
The multipliers $(\lambda,\mu)$ are updated whenever $x$ is updated.

\begin{remark} \label{remark1}
Note by the $z$-update  \eqref{eq:z_update} that if $\alpha$ is large enough, minimizing $\mathcal{L}_\beta$ in $z$ will tend to make $\|z_{k+1}\|$ small enough, even if $\mu$ and $\lambda$ are somewhat arbitrary. Furthermore, the negative quadratic term $-\frac{1}{2\rho}\| \lambda - \mu_{k+1} \|^2$ does not allow for $\lambda_{k+1}$ to deviate far from the bounded $\mu_{k+1}$ by the $\lambda$-update \eqref{eq:lambda_update}. Hence, there must exist large enough $\alpha$ such that $\| z_{k+1} \| \leq \alpha \|z_{k+1} - z_k \|$ for all $k\geq0$. This in turn leads to $\nabla_\mu \mathcal{L}_\beta (x_{k+1}, z_{k+1}, \lambda_{k+1}, \mu_{k+1}) = 0$ as $k \rightarrow \infty$ if we can show that $\text{lim}_{k \rightarrow \infty}\|x_{k+1} - x_k \|=0$. That is, 
\[
\frac{1}{\rho} \| \lambda_{k+1} - \mu_{k+1} \| \leq \frac{\alpha^2}{\rho} \|z_{k+1} - z_k \| \leq \alpha L_c\| x_{k+1} - x_k \| \rightarrow 0
\]
 as $k  \rightarrow \infty$. The last inequality comes from $\alpha z_k = \rho c(x_k)$ and $L_c$-Lipschitz continuity of $c$.
\end{remark}

\begin{remark} \label{remark2}
 When updating the multiplier $\mu$, it is important to choose the reduction ratio $r$ close to 1 (e.g. 0.999 or even closer to 1). Choosing a small value of  $r$ will cause the iterate of multiplier $\mu_k$ to reach a point quickly in a small number of iterations, which in turn may cause the multiplier $\lambda_k$ to stay far away from the multiplier $\lambda^\ast$ satisfying the KKT conditions \eqref{eq:def_FOC}.
\end{remark}

 \section{Convergence Analysis} \label{sec:convergence}
 
In this section, we establish the convergence results of Algorithm \ref{algorithm1}. We prove that the sequence generated by Algorithm 1 has limit points and any limit point is a KKT point of problem \eqref{eq:op}. To analyze the convergence of Algorithm \ref{algorithm1}, 
we need to recall the well-known descent Lemma, which is a direct consequence of Assumptions \ref{assumption_lipschitz_i} and \ref{assumption_lipschitz_ii}.
\begin{lemma}[{\cite[Proposition A.24]{bertsekas1999nonlinear}}]\label{lem_descent} 
Let  Assumptions \ref{assumption_lipschitz_i}$-$\ref{assumption_lipschitz_iii} hold. Then for any fixed  $(z,\lambda,\mu)$, $\nabla_x \mathcal{L}_\beta$ is Lipschitz continuous with constant $L_p > 0$. We have that for any $ x_{1},x_{2}\in {X}$
 \[
  \mathcal{L}_\beta(x_1) \leq\mathcal{L}_\beta(x_2) + \left\langle \nabla_x\mathcal{L}_\beta(x_2), x_1 - x_2 \right\rangle  +\frac{L_p}{2}\| x_1 - x_2\|^2,
 \]
Here we omit fixed $(z,\lambda,\mu)$ for simplicity.
\end{lemma}

Let us observe some relations on sequences $\{\lambda_k\}_{k \in \mathbb{N}}$, $\{\mu_k\}_{k \in \mathbb{N}}$, and $\{ x_k \}_{k \in \mathbb{N}}$ generated by Algorithm \ref{algorithm1}, which are important for deriving the nonincreasing property of $\mathcal{L}_\beta$.

\begin{lemma}\label{lem_iter_rel} 
Let  $\{(x_k,z_k, \lambda_k, \mu_k)\}_{k \in \mathbb{N}}$ be the sequence  generated by Algorithm \ref{algorithm1}. Then, the following hold:
\begin{align}
\| \mu_{k+1} - \mu_{k} \|^2 & \leq ({\gamma_k}/{\rho}) \| \lambda_k - \mu_k \|^2 \leq \delta_k,    \label{eq:lem_iter_rel_1} \\
\| \lambda_{k+1}-\lambda_{k} \|^2 & \leq 2 \rho^2 L_c^2 \| x_{k+1} - x_k \|^2 + 2\delta_k, \label{eq:lem_iter_rel_2} \\
\| \mu_{k+1} - \lambda_{k}\|^2 & =
\left(1 - {\gamma_k}/{\rho}\right)^2 \| \lambda_k - \mu_k \|^2,    \label{eq:lem_iter_rel_3}
\end{align}
where $\rho=\frac{\alpha}{1+\alpha\beta}$ and $L_c$ denotes the Lipschitz constant of $c$.
\end{lemma}

\begin{proof}
By the $\mu$-update \eqref{eq:mu_update}, we have 
$\| \mu_{k+1} - \mu_k \| =  \frac{\gamma_k}{\rho}  \| \lambda_k - \mu_k \|.$
Since $\frac{\gamma_k}{\rho} = \frac{\delta_k}{\| \mathbf{\lambda}_k - \mathbf{\mu}_k \|^2 + 1} \leq 1$, implying that  $\frac{\gamma_k^2}{\rho^2} \leq \frac{\gamma_k}{\rho}  \leq 1 $, the first inequality in \eqref{eq:lem_iter_rel_1} holds. By the definition of $ \gamma_k$, we  deduce the second inequality in \eqref{eq:lem_iter_rel_1}: 
\begin{equation} \label{eq:lem_iterates_relations_1_1}
	\frac{\gamma_k}{\rho} \| \lambda_k - \mu_k \|^2 = 
	\frac{1}{\rho} \cdot
	\frac{\rho \delta_k}{\| \lambda_k - \mu_k \|^2 + 1} \cdot \| \lambda_k - \mu_k \|^2 
	\leq \delta_k.  \notag
\end{equation}

From the $\lambda$-update step \eqref{eq:lambda_update}, we know that $\lambda_{k+1}- \lambda_{k} 
= (\mu_{k+1}-\mu_{k}) + {\rho}( c(x_{k+1}) -c(x_{k}) )$. 
Using the triangle inequality and the  $L_c$-Lipschitz continuity of $c$, we have
\begin{equation}
	\| \lambda_{k+1}- \lambda_k \| 
	\leq
	\| \mu_{k+1} - \mu_k \| + \rho L_c \| x_{k+1} - x_k \|.   \label{eq:lem_iterates_relations_e2} \notag
\end{equation}
We also have that $\| \mu_{k+1} - \mu_k \| =\frac{\gamma_k}{\rho} \| \lambda_k - \mu_k \| \leq \delta_k$, 
Using the facts that $\delta_k^2 \leq \delta_k \leq 1$ and $(a+b)^2 \leq 2a^2+2b^2$ for any $a,b \in \mathbb{R}$, we obtain the desired relation \eqref{eq:lem_iter_rel_2}: 
\begin{equation}
\begin{aligned}
	\| \lambda_{k+1}- \lambda_k \|^2 
	& \leq
	2 \rho^2 L_c^2 \| x_{k+1} - x_k \|^2 + 2\| \mu_{k+1} - \mu_k \|^2,    \notag \\
	& \leq 2 \rho^2 L_c^2 \| x_{k+1} - x_k \|^2 + 2 \delta_k. \notag
\end{aligned}
\end{equation}
By subtracting $\mu_{k+1}$ from $\lambda_k$, we have 
\[
\begin{aligned}
\| \lambda_k - \mu_{k+1} \| 
& = \| \lambda_k - \mu_k - \frac{\gamma_k}{\rho}(\lambda_k -\mu_k) \| \notag \\
&
= \left(1 - \frac{\gamma_k}{\rho}\right) \| \lambda_k - \mu_k \|.
\end{aligned}
\]
Squaring both sides of the above inequality yields the desired relation \eqref{eq:lem_iter_rel_3}.
\end{proof}

We now show that the sequence $\{\mathcal{L}_\beta(x_k,z_k,\lambda_k,\mu_k)\}_{k \in \mathbb{N}}$ is approximately nonincreasing, namely {\em approximate sufficient decrease property of $\{ \mathcal{L}_\beta(x_k,z_k,\lambda_k,\mu_k)\}_{k \in \mathbb{N}}$}; see \cite{gur2023convergent} for details.
 
\begin{lemma}\label{thm_sufficient_decrease}
Let $\left\{(x_k,z_k,\lambda_k,\mu_k) \right\}_{k \in \mathbb{N}}$ be the sequence generated by Algorithm \ref{algorithm1}. Then, we have that for any $k \in \mathbb{N}$ 
\begin{equation} \label{eq:sufficient_decrease}
\begin{aligned}
\mathcal{L}_\beta(x_{k+1},z_{k+1},\lambda_{k+1},\mu_{k+1})  \leq \mathcal{L}_\beta(x_k,z_k,\lambda_k,\mu_k) -
   \frac{1}{2} \left(\eta - L_p  - 2\rho L_c^2  \right) \| x_{k+1} - x_k \|^2 + \hat{\delta}_k,   
\end{aligned}
\end{equation}
where $\hat{\delta}_k = {2\delta_k}/{\rho}$. In particular, if a sufficiently large $\eta$ is chosen such that  $\eta > L_p +2\rho L_c^2$, the sequence  $\{ \mathcal{L}_\beta(x_k,z_k,\lambda_k,\mu_k) \}_{k \in \mathbb{N}}$ is approximately nonincreasing.
\end{lemma}
 
\begin{proof}
Notice first that 
\[
\begin{aligned} 
   \mathcal{L}_\beta(x_k,z_k,\lambda_k,\mu_k)    
   & = f(x_k)  + \left\langle \lambda_k, c(x_k) \right\rangle  - \left\langle \lambda_k - \mu_k,z_k \right\rangle + \frac{\alpha}{2} \| z_k \|^2 - \frac{\beta}{2} \| \lambda_k - \mu_k \|^2  \\   
   & = f(x_k)  + \left\langle \lambda_k, c(x_k) \right\rangle - \frac{1}{2\rho} \| \lambda_k - \mu_k \|^2  \\
   & = \mathcal{L}_\beta(x_k,\hat{z}(\lambda_k,\mu_k),\lambda_k,\mu_k),     
\end{aligned}
\]
and $$\mathcal{L}_\beta(x_{k+1},z_k,\lambda_k,\mu_k) = \mathcal{L}_\beta(x_{k+1},\hat{z}(\lambda_k,\mu_k),\lambda_k,\mu_k).$$ Then, the difference of two successive sequences of $ \mathcal{L}_\beta$ can be divided into two parts as follows:
\begin{equation}\label{eq:lem_sufficient_decrease_p1}
\begin{aligned}
   & \mathcal{L}_\beta(x_{k+1},z_{k+1},\lambda_{k+1},\mu_{k+1}) -        \mathcal{L}_\beta(x_k,z_k,\lambda_k,\mu_k)    \\
   & = \left[  \mathcal{L}_\beta(x_{k+1},z_k,\lambda_k,\mu_k) -        \mathcal{L}_\beta(x_k,z_k,\lambda_k,\mu_k) \right]  \\
   & + \left[ \mathcal{L}_\beta(x_{k+1},\hat{z}(\lambda_{k+1},\mu_{k+1}),\lambda_{k+1},\mu_{k+1}) - \mathcal{L}_\beta(x_{k+1},\hat{z}(\lambda_k,\mu_k),\lambda_k,\mu_k) \right].  
\end{aligned}
\end{equation}
For the first part, by Lemma \ref{lem_descent}, we get
\[
\mathcal{L}_\beta(x_{k+1}) \leq \mathcal{L}_\beta(x_k) + \left\langle \nabla_x \mathcal{L}_\beta(x_k), x_{k+1} - x_k\right\rangle + \frac{L_p}{2}\| x_{k+1} - x_k \|^2.
\]
Here, $(z_k,\lambda_k,\mu_k )$ is omitted for simplicity. By the definition of $x_{k+1}=\mathrm{argmin}_{x \in X}  \widehat{\mathcal{L}}_\beta(x,z_k,\lambda_k,\mu_k;x_k)$, we have
\[
\begin{aligned}
 	\widehat{\mathcal{L}}_\beta(x_{k+1}; x_k) 
 	& = \mathcal{L}_\beta(x_k)
 	+\left\langle \nabla_{x}\mathcal{L}_\beta(x_k), x_{k+1} - x_k \right\rangle + \frac{\eta}{2}\| x_{k+1} - x_k \|^2 \\ & \leq \widehat{\mathcal{L}}_\beta(x_k;x_k) = \mathcal{L}_\beta(x_k),
\end{aligned}
\]
which implies $\left\langle \nabla_{x}\mathcal{L}_\beta(x_k), x_{k+1} - x_k \right\rangle \leq -\frac{\eta}{2}\| x_{k+1}-x_{k}\|^2$. Thus,
\begin{equation} \label{eq:lem_sufficient_decrease_p2}
\mathcal{L}_\beta(x_{k+1}) - \mathcal{L}_\beta(x_k) \leq - \frac{1}{2}\left(\eta-L_p\right) \|x_{k+1} - x_k\|^2. 
\end{equation}
Now, we derive an upper bound for the second part on the RHS of \eqref{eq:lem_sufficient_decrease_p1}. We start by noting that
\begin{equation} \label{eq:lem_sufficient_decrease_d1}
\begin{aligned}
& \mathcal{L}_\beta(x_{k+1},\hat{z}(\lambda_{k+1},\mu_{k+1}),\lambda_{k+1},\mu_{k+1}) - \mathcal{L}_\beta(x_{k+1},\hat{z}(\lambda_{k},\mu_{k}),\lambda_k,\mu_k)           \\ 
& = \frac{1}{\rho} \left\langle \lambda_{k+1} - \lambda_{k}, c(x_{k+1})  \right\rangle
- \frac{1}{2\rho}  \left(\| \lambda_{k+1} - \mu_{k+1} \|^2 -  \| \lambda_k - \mu_k \|^2\right). \notag
\end{aligned}
\end{equation}
Using the facts that $c(x_{k+1})=\frac{1}{\rho}(\lambda_{k+1}-\mu_{k+1})$ and $\left\langle a, b \right\rangle =\frac{1}{2} \| a \|^2 + \frac{1}{2} \| b \|^2 - \frac{1}{2} \| a - b \|^2,$ we have that 
$$\frac{1}{\rho}\left\langle \lambda_{k+1} - \lambda_{k}, \lambda_{k+1} - \mu_{k+1}  \right\rangle  
 = \frac{1}{2\rho} \left(\| \lambda_{k+1}-\lambda_k \|^2 +  \| \lambda_{k+1} -\mu_{k+1} \|^2 - \| \mu_{k+1} -\lambda_k \|^2 \right).$$ Thus,
\begin{align}
& 
\mathcal{L}_\beta(x_{k+1},\hat{z}(\lambda_{k+1},\mu_{k+1}),\lambda_{k+1},\mu_{k+1}) - \mathcal{L}_\beta(x_{k+1},\hat{z}(\lambda_{k},\mu_{k}),\lambda_k,\mu_k)   \notag\\
& \overset{(a)}{\leq}
\frac{1}{2\rho}  \left( 2\rho^2 L_c^2    \| x_{k+1} - x_k \|^2 +2 \delta_k \right) 
+ \frac{1}{2\rho}\left( 1- \left(1-\frac{\gamma_k}{\rho} \right)^2\right) \| \lambda_k - \mu_k \|^2 \notag \\
& =
\rho L_c^2 \| x_{k+1} - x_k \|^2 + \frac{\delta_k}{\rho}
+ \frac{1}{2\rho}\left( \frac{2\gamma_k}{\rho} - \frac{\gamma_k^2}{\rho^2}\right) \| \lambda_k - \mu_k \|^2 \notag \\
&  \overset{(b)}{\leq}
\rho L_c^2 \| x_{k+1} - x_k \|^2 + \frac{\delta_k}{\rho}
+ \frac{\delta_k}{\rho} , \label{eq:lem_sufficient_decrease_d4}
\end{align}	
where $(a)$ is from \eqref{eq:lem_iter_rel_2} and \eqref{eq:lem_iter_rel_3} in Lemma \ref{lem_iter_rel}, $\| \mu_{k+1} - \lambda_k \|^2 = \left(1 - \gamma_k / \rho \right)^2 \| \lambda_{k} - \mu_k \|^2$ and $\| \lambda_{k+1}-\lambda_{k} \|^2 \leq  2\rho^2 L_c^2  \| x_{k+1}-x_{k} \|^2 + 2\delta_k$, and  $(b)$ holds by
$ \frac{\gamma_k}{\rho} \| \lambda_k - \mu_k \|^2 \leq  \delta_k$. 
Combining  \eqref{eq:lem_sufficient_decrease_p2} and \eqref{eq:lem_sufficient_decrease_d4} yields  the desired result \eqref{eq:sufficient_decrease}.
Therefore, $\{\mathcal{L}_\beta(x_k,z_k,\lambda_k,\mu_k)\}_{k \in \mathbb{N}}$ is nonincreasing if $\eta$ is chosen such that $\eta > L_p  + 2 \rho L_c^2$. 
\end{proof}
We provide the key properties that $\{ \mathcal{L}_\beta(x_k,z_k,\lambda_k,\mu_k) \}_{k \in \mathbb{N}}$ is convergent and $\{(x_k,z_k,\lambda_k,\mu_k) \}_{k \in \mathbb{N}}$ is bounded.
\renewcommand\thetheorem{1}
\begin{theorem} \label{lem_Lagrangian_converge}
Let $\{(x_k,z_k,\lambda_k,\mu_k)_{k \in \mathbb{N}}$ be the sequence generated by Algorithm \ref{algorithm1}. Then, the sequence $\{ \mathcal{L}_\beta(x_k,z_k,\lambda_k,\mu_k) \}_{k \in \mathbb{N}}$ is convergent, i.e., 
\[
\underset{k \rightarrow \infty}{\text{lim}}\mathcal{L}_\beta(x_k,z_k,\lambda_k,\mu_k) 
:= \underline{\mathcal{L}_\beta}  > -\infty.
\]
Furthermore, the sequence $\{(x_k,z_k,\lambda_k,\mu_k) \}_{k \in \mathbb{N}}$ is bounded.
\end{theorem}

\begin{proof}
By $\mathcal{L}_\beta(x_{k+1},z_{k+1},\lambda_{k+1},\mu_{k+1}) =f(x_{k+1})+\left\langle \lambda_{k+1}, c(x_{k+1}) \right\rangle - \frac{1}{2\rho} \| \lambda_{k+1} - \mu_{k+1} \|^2$ and $\lambda_{k+1}=\mu_{k+1}+\rho c(x_{k+1})$, we have
\begin{align} 
\mathcal{L}_\beta(x_{k+1},z_{k+1},\lambda_{k+1},\mu_{k+1}) 
& =  f(x_{k+1}) + \left\langle \mu_{k+1}, c(x_{k+1}) \right\rangle + \frac{\rho}{2}\| c(x_{k+1}) \|^2  \notag \\
&  \geq  f(x_{k+1}) - \frac{1}{2t}\| \mu_{k+1} \|^2 - \frac{t}{2}\| c(x_{k+1}) \|^2 + \frac{\rho}{2}\| c(x_{k+1}) \|^2, \notag
\end{align}
where we used Young's inequality $\left\langle a, b \right\rangle \geq - \frac{1}{2t} \| a \|^2 - \frac{t}{2} \| b \|^2,$ $\forall a,b \in \mathbb{R}^m$ and any $t>0$. Combining the above inequality and \eqref{eq:sufficient_decrease}, we obtain with the choice of $t=\rho$ that
\begin{align}
f(x_{k+1}) - \frac{1}{2\rho}\| \mu_{k+1} \|^2 & \leq \mathcal{L}_\beta(x_{k+1},z_{k+1},\lambda_{k+1},\mu_{k+1}) \notag \\ 
& \leq \mathcal{L}_\beta(x_k,z_k,\lambda_k,\mu_k)   -
\frac{1}{2} \left(\eta - L_p  - 2\rho L_c^2  \right) \| x_{k+1} - x_k \|^2 + \hat{\delta}_k \notag \\
& <\infty.   \notag
\end{align}
Since $f(x)$ is lower bouned \ref{assumption_coercive} and  $\{\mu_k\}_{k \in \mathbb{N}}$ is bounded (Lemma \ref{lem_boundness_art_dual}), we have that $\mathcal{L}_\beta(x_k,z_k,\mu_k,\lambda_k) > -\infty$. Hence, by Lemma \ref{thm_sufficient_decrease} and the coercivity of $f$ \ref{assumption_coercive}, $\{\mathcal{L}_\beta(x_k,z_k,\lambda_k,\mu_k)\}_{k \in \mathbb{N}}$ is convergent to a finite value $\underline{\mathcal{L}_\beta}$ and  $\{  x_k \}_{k \in \mathbb{N}}$ is bounded. From  $\lambda_{k+1}=\mu_{k+1}+\rho c(x_{k+1})$, we have that $\{ \lambda_k\}_{k \in \mathbb{N}}$ is bounded. It also follows from the $z$-update \eqref{eq:z_update} that $\{ z_k\}_{k \in \mathbb{N}}$ is bounded.  
\end{proof}
 
We note that if the $X$ is compact, the lower boundedness of $\{\mathcal{L}_\beta(x_k,z_k,\lambda_k,\mu_k)\}_{k \in \mathbb{N}}$ and boundedness of $\{ x_k \}_{k \in \mathbb{N}}$ are readily satisfied with Lemma \ref{lem_boundness_art_dual}. Equipped with Lemma \ref{thm_sufficient_decrease} and Theorem \ref{lem_Lagrangian_converge}, we immediately obtain the following result.
\renewcommand\thetheorem{5}\
\begin{lemma} \label{thm_boundedness}
Let $\{(x_k,z_k,\lambda_{k},\mu_k)\}_{k \in \mathbb{N}}$ be the sequence generated by Algorithm \ref{algorithm1}. Then it holds  
$\sum_{k=1}^{\infty}\| x_{k+1} - x_k \|^2 < \infty,$ 
and hence
 \begin{equation} \label{eq:thm_boundedness_result2}
 \begin{aligned}
 	 \underset{k\rightarrow\infty} \lim \| x_{k+1} - x_k \|  = 0, \ \ 
 	\underset{k\rightarrow\infty} \lim \| z_{k+1} - z_k \|  = 0, \ \
 	 \underset{k\rightarrow\infty} \lim \| \lambda_{k+1}-\lambda_k \| = 0,\ \
 	\underset{k\rightarrow\infty} \lim \| \mu_{k+1}- \mu_k \| = 0.
 \end{aligned}
 \end{equation}
\end{lemma}
\begin{proof}
Write $\mathcal{L}_\beta^k = \mathcal{L}_\beta(x_k,z_k,\lambda_k,\mu_k)$ for simplicity. Invoking Lemma \ref{thm_sufficient_decrease}, we get that for all $k\geq1$
\begin{equation}\label{eq:prop_asymp_regular_1}
   C\| x_{k+1} - x_{k} \|^2 \leq \mathcal{L}_\beta^k - \mathcal{L}_\beta^{k+1} + \hat{\delta}_k,
\end{equation}
where $C:=\frac{1}{2}\left( \eta -  L_p - 2\rho L_c^2  \right) > 0$ and $\hat{\delta}_k=\frac{2\delta_k}{\rho}$. Summing \eqref{eq:prop_asymp_regular_1} over $k=1$ to $k=K$, we obtain 
\begin{equation}
   C \sum_{k=1}^K \| x_{k+1} - x_k \|^2  \leq \mathcal{L}_\beta^1 - \mathcal{L}_\beta^{K+1} + \sum_{k=1}^K \hat{\delta}_k. \notag
\end{equation}
Letting $K \rightarrow \infty$ yields 
$$\sum_{k=1}^{\infty}\| x_{k+1} - x_k \|^2 < \infty,$$
which holds by the convergence of $\{\mathcal{L}_\beta^k\}$ and $\sum_{k=1}^K \delta_k <  \infty$. 
It follows from $\alpha z_k = \rho c(x_k)$ that $$\sum_{k=1}^{\infty}\| z_{k+1} - z_k \|^2 \leq \frac{\rho^2 L_c^2}{\alpha^2}\sum_{k=1}^{\infty}\| x_{k+1} - x_k \|^2< \infty.$$
From Lemma \ref{lem_iter_rel}, $\mu$-update, and $\lambda$-update steps, it readily follows that 
$$\sum_{k=1}^{\infty}\| \mu_{k+1} -\mu_k \|^2   < \infty \  \ \text{and} \  \
\sum_{k=1}^{\infty}\| \lambda_{k+1} - \lambda_k \|^2   < \infty.$$
Hence, the desired results in \eqref{eq:thm_boundedness_result2} follow. 
\end{proof}

Having the preceding properties of Algorithm \ref{algorithm1}, we prove our main result, which asserts that the sequence generated by Algorithm \ref{algorithm1} converges to a KKT point of problem \eqref{eq:op}. 
\renewcommand\thetheorem{2}
\begin{theorem}\label{thm_sub_convergence}
Assume that there exists $\lambda^\ast \in \mathbb{R}^m$ such that the KKT conditions of problem \eqref{eq:op} hold:
\begin{equation} \label{eq:thm_original_KKT} 
 0 \in \nabla f(x^\ast) +\nabla c(x^\ast) \lambda^\ast + \mathcal{N}_X( x^\ast), \qquad c(x^\ast)=0,
\end{equation}
and Assumptions \ref{assumption_lipschitz_i}$-$\ref{assumption_coercive} are satisfied. Let $\{ (x_k,z_k,\lambda_k,\mu_k) \}_{k \in \mathbb{N}}$ be the sequence generated by Algorithm \ref{algorithm1}. Then,  any limit point of $\{(x_k,z_k,\lambda_k,\mu_k)\}_{k \in \mathbb{N}}$ is a KKT point of problem \eqref{eq:op}, i.e., $\{(x_k,z_k,\lambda_k,\mu_k) \}_{k \in \mathbb{N}}$ has a limit point $(\overline{x},\overline{z},\overline{\lambda},\overline{\mu})$, in which
$(\overline{x},\overline{\lambda})$ satisfies  \eqref{eq:thm_original_KKT}.
\end{theorem}

\begin{proof}
Since  $\{ (x_k,z_k,\lambda_k,\mu_k) \}_{k \in \mathbb{N}}$ is bounded, there is at least one limit point. Let $(\overline{{x}}, \overline{z},\overline{\lambda},\overline{\mu})$ be a limit point of $\{(x_k, z_k,\lambda_k,\mu_k)\}_{k \in \mathbb{N}}$, and let $\{ (x_{k_{j}}, z_{k_{j}},\lambda_{k_{j}},\mu_{k_{j}})\}_{j \in \mathbb{N}}$ be a subsequence converging to $(\overline{x},\overline{z},\overline{\lambda},\overline{\mu})$ as $j \rightarrow \infty$. From Lemma \ref{thm_boundedness},
it also follows that $\{ (x_{k_j+1},z_{k_j+1},\lambda_{k_j+1},\mu_{k_j+1}) \}_{j \in \mathbb{N}} \rightarrow (\overline{x},\overline{z},\overline{\lambda},\overline{\mu})  \text{ as } j \rightarrow \infty$.  
By the continuity of $f$ and $c$, we have 
$\text{lim}_{j \rightarrow \infty} \mathcal{L}_\beta(x_{k_j},z_{k_j},\lambda_{k_j},\mu_{k_j})
=\mathcal{L}_\beta(\overline{x},\overline{z},\overline{\lambda},\overline{\mu})$. Hence, 
we have
\begin{equation} \label{eq:project_gradient_x}
	\overline{x} =\text{P}_X \left[ \overline{x} - \eta^{-1} \nabla_x \mathcal{L}_\beta(\overline{x},\overline{z},\overline{\lambda},\overline{\mu}) \right], \notag
\end{equation}
which is equivalent to the inclusion \cite[Theorem 6.12]{rockafellar2009variational}:
\[0 \in \nabla f(\overline{x}) + \nabla c(\overline{x}) \overline{\lambda}  +  \mathcal{N}_X(\overline{x}).\]
By Remark \ref{remark1} and $\lim_{k\rightarrow\infty} \| x_{k+1} - x_k \|  = 0$ (Lemma \ref{thm_boundedness}), we obtain that for $\alpha>0$ large enough and $\beta>0$, 
\begin{align}
	 \|c(x_{k+1}) \| & = \frac{1}{\rho} \| \lambda_{k+1} - \mu_{k+1} \|  = \frac{\alpha}{\rho} \| z_{k+1} \|
     \leq \frac{\alpha^2}{\rho} \| z_{k+1} - z_k \| \leq \alpha L_c \| x_{k+1} - x_k \| \rightarrow 0 \ \ \text{as} \ \ k \rightarrow \infty. \notag 
\end{align}
This, along with the updating rules for $z_{k+1}$ and $\lambda_{k+1}$ implies
\[
\overline{z} = \frac{1}{\alpha}(\overline{\lambda} - \overline{\mu}) = 0  \ \ \text{and}  \ \
\overline{\lambda}-\overline{\mu} = \rho c(\overline{x})=0.
\]
Therefore, we obtain
\begin{align}
 0 \in \nabla f(\overline{x}) + \nabla c(\overline{x}) \overline{\lambda} + \mathcal{N}_{{X}}( \overline{x}), \qquad
 c(\overline{x})=0, \notag 
\end{align}
implying that the limit point $(\overline{x},0,\overline{\lambda},\overline{\lambda})$ of  $\{ ({x}_{k}, z_k,\lambda_k,\mu_k) \}_{k \in \mathbb{N}}$ is a KKT point of problem \eqref{eq:op}. 
\end{proof}

\section{Numerical Experiments} \label{sec:experiments}

We conduct preliminary experiments to illustrate the validity of Algorithm \ref{algorithm1}. The performance of the algorithm is evaluated on three test instances where LICQ does not hold. We examine the behaviors of Algorithm \ref{algorithm1} using small fixed step sizes in the experiments. We report the quantities as measures of optimality and feasibility:
$\left\| x_k - \mathrm{P}_X[x_k - \nabla_x \mathcal{L}_\beta (x_k,z_k,\lambda_k, \mu_k)] \right\|$  and $\rho^{-1}\| \lambda_k -\mu_k\|.$
Clearly,  $\rho^{-1}\| \lambda_k -\mu_k\| = 0 $ when $c(x_k)=0$ (see the $\lambda$-update step \eqref{eq:lambda_update}). 
\begin{figure*}[t!]
\centering
\begin{multicols}{3}
\includegraphics[scale=0.38]{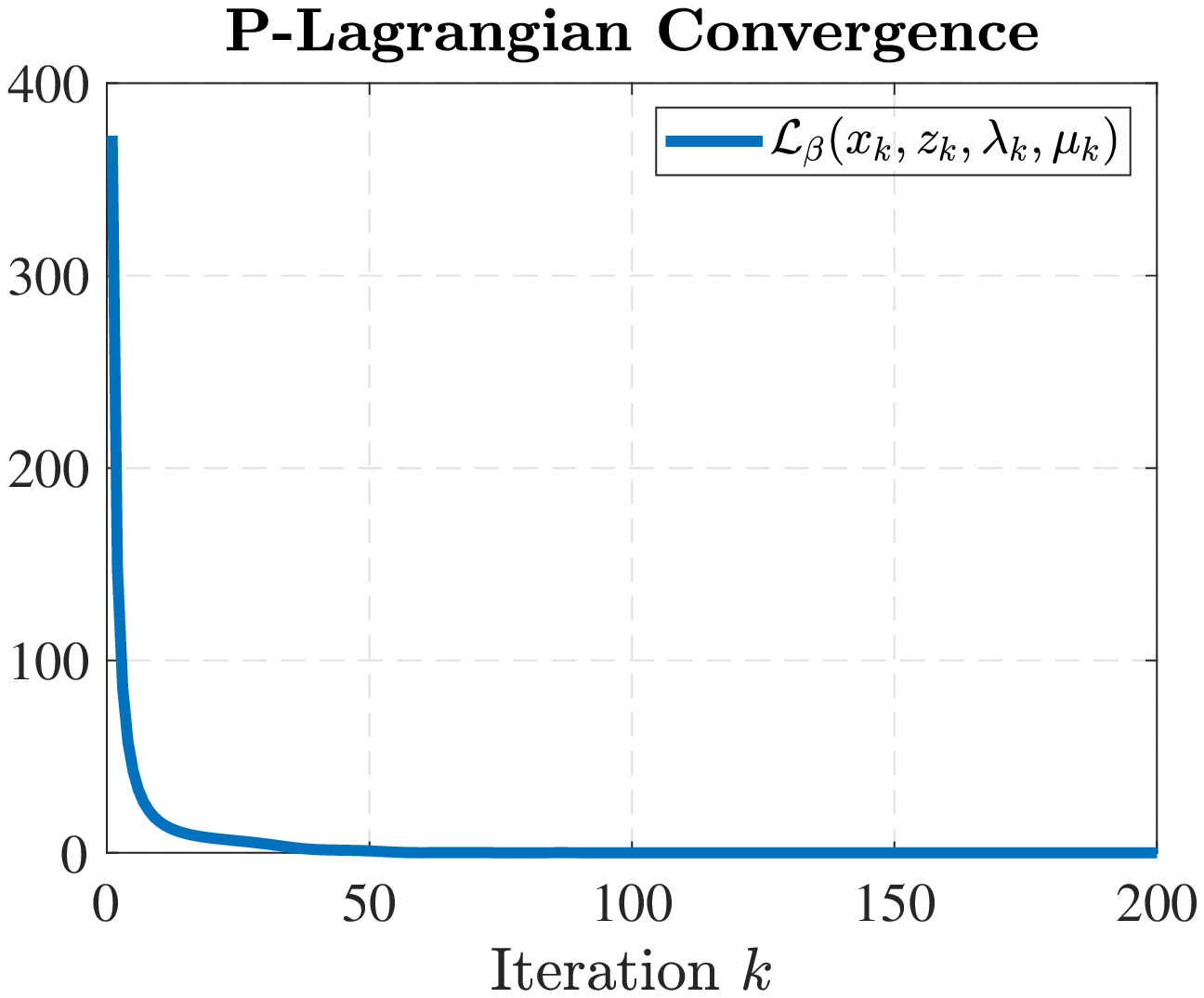}

\includegraphics[scale=0.38]{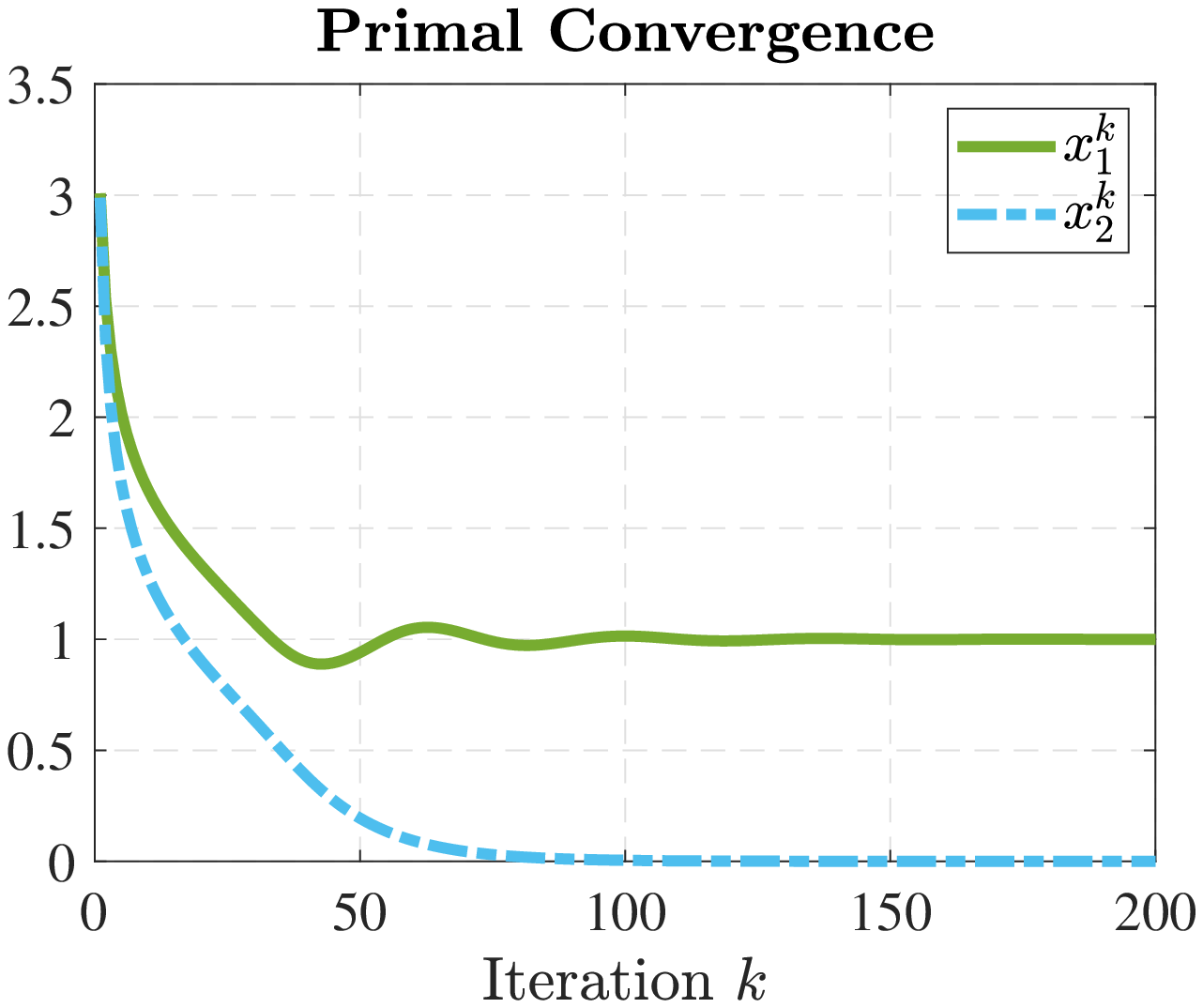}  

\includegraphics[scale=0.38]{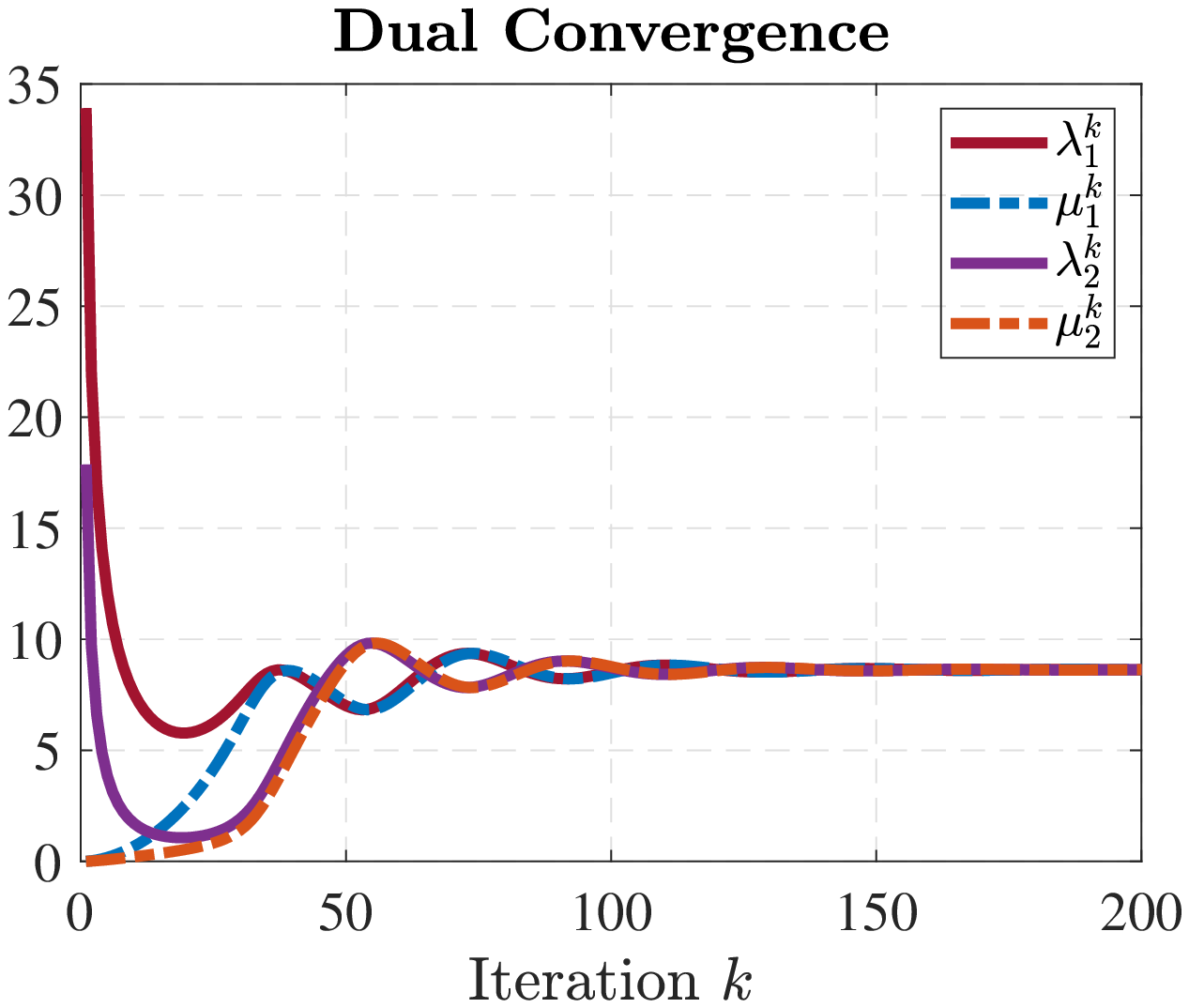}
\end{multicols}
\begin{multicols}{2}
\hspace{0.2in}\includegraphics[scale=0.38]{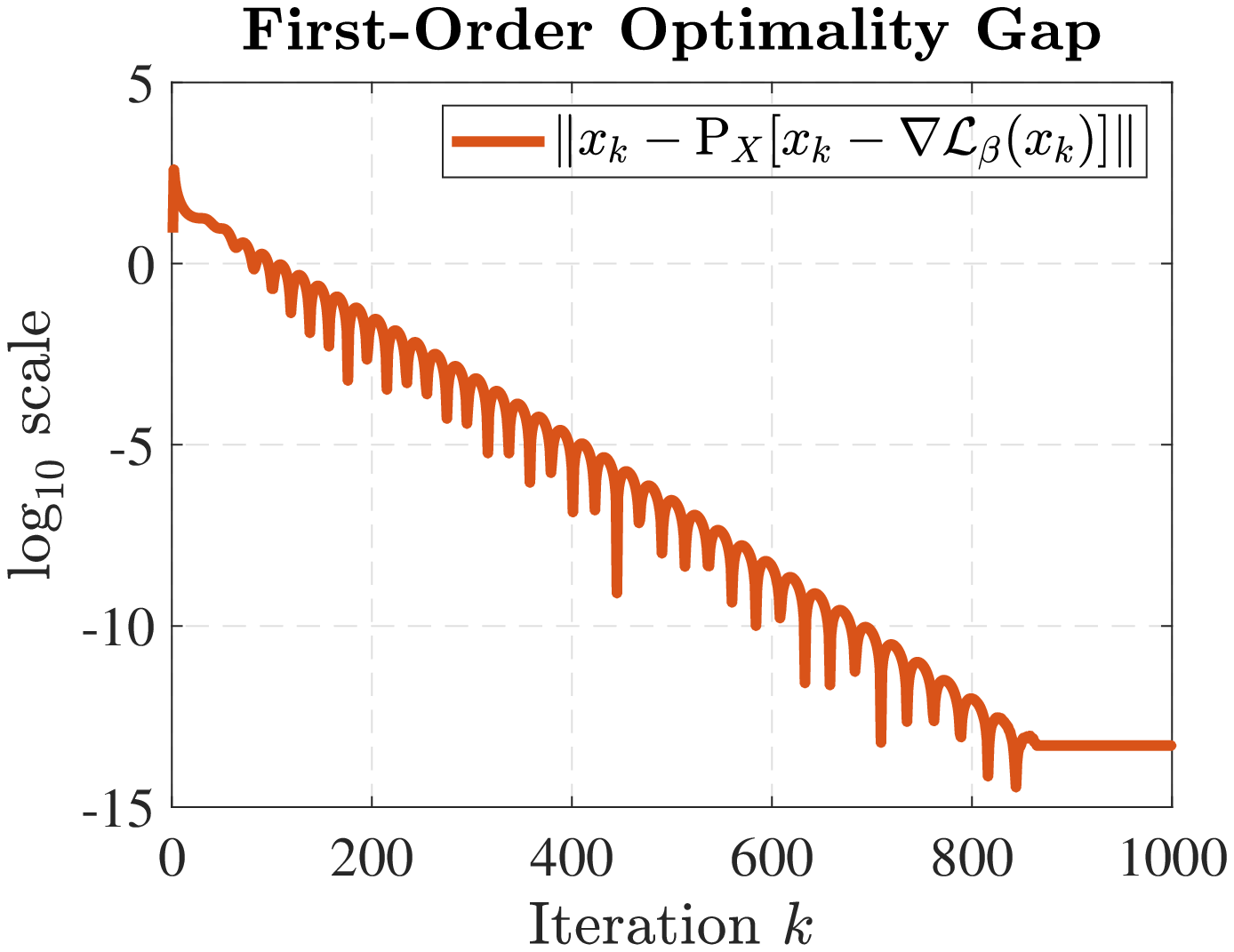}

\includegraphics[scale=0.38]{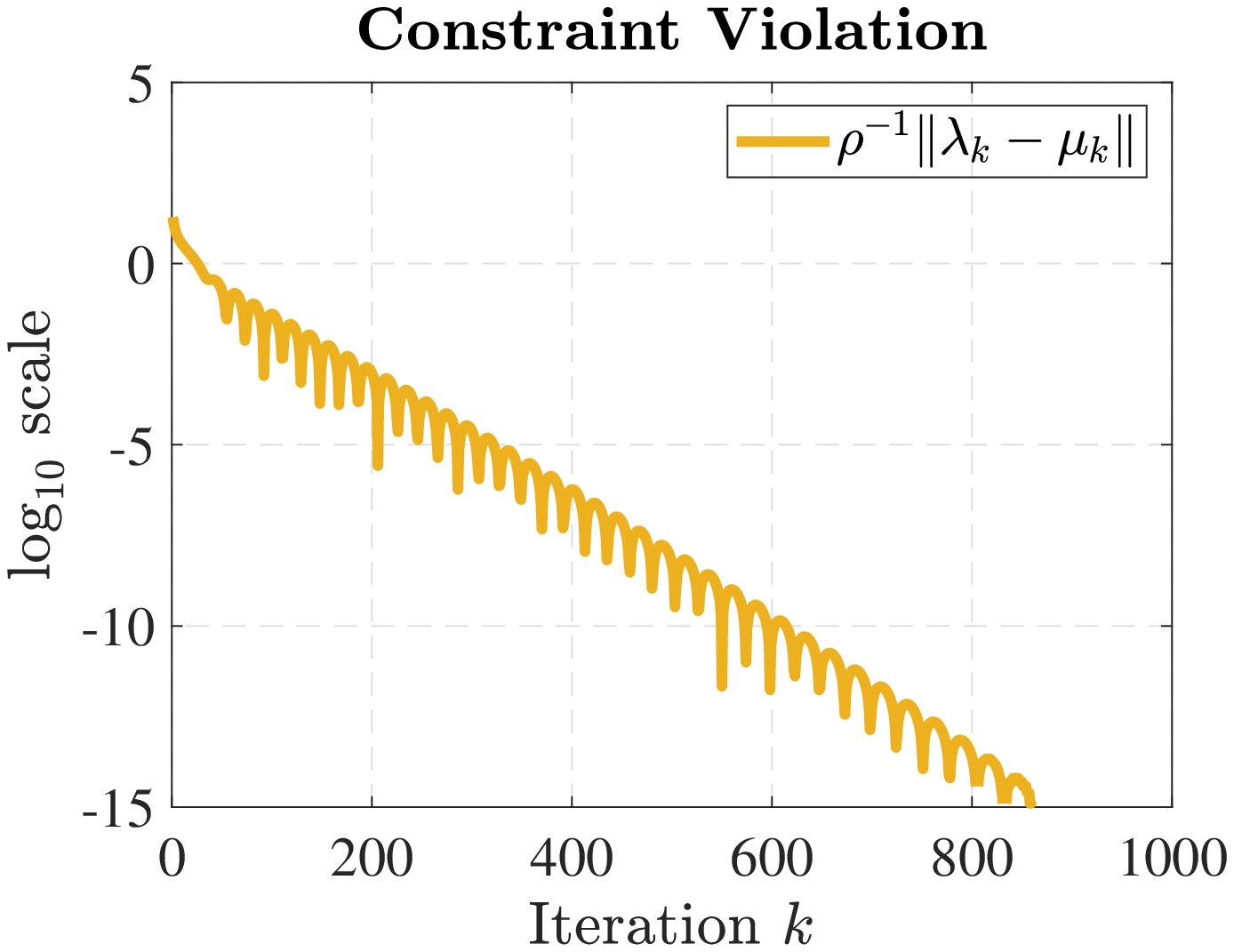} 
\end{multicols}

\caption{Convergence behavior of Algorithm \ref{algorithm1} applied to Example \ref{example1}. A small constant step size $\eta^{-1}=0.002$ is used in this experiment. The rest of the parameters are set to $\alpha=2000$,  $\beta=0.5$, and  $\delta_0=1$ with $r=0.999$.} \label{fig:ex1}
\end{figure*}

\begin{example}\label{example1}
Consider the simple nonconvex problem:
\begin{equation} 
\begin{aligned}
	\underset{-3 \leq x_1, x_2 \leq 3}{\mathrm{min}} & \ \   f(x)=-(x_1-1)^2 + x_2^2  \\
	\mathrm{s.\:t.} & \ \ c_1(x) = x_1^2+x_2^2-1=0 \\
			      & \ \ c_2(x) = (x_1-2)^2 + x_2^2-1=0. \notag
\end{aligned} 
\end{equation}
\end{example}
At the optimal solution (and only feasible solution) $(x_1^\ast, x_2^\ast) = (1,0)$, LICQ does not hold. Figure \ref{fig:ex1} illustrates the convergence behavior of Algorithm \ref{algorithm1}  and describes the setting of the parameters. We see that starting from $x_0=(3,3)$ and $(z_0,\lambda_0,\mu_0)=({0,0,0})$, Algorithm 1 converges to the optimal solution. It is also shown that the iterates of multipliers remain bounded and converge. 

\begin{figure*}[t!] 
\centering
\begin{multicols}{3}
\includegraphics[scale=0.38]{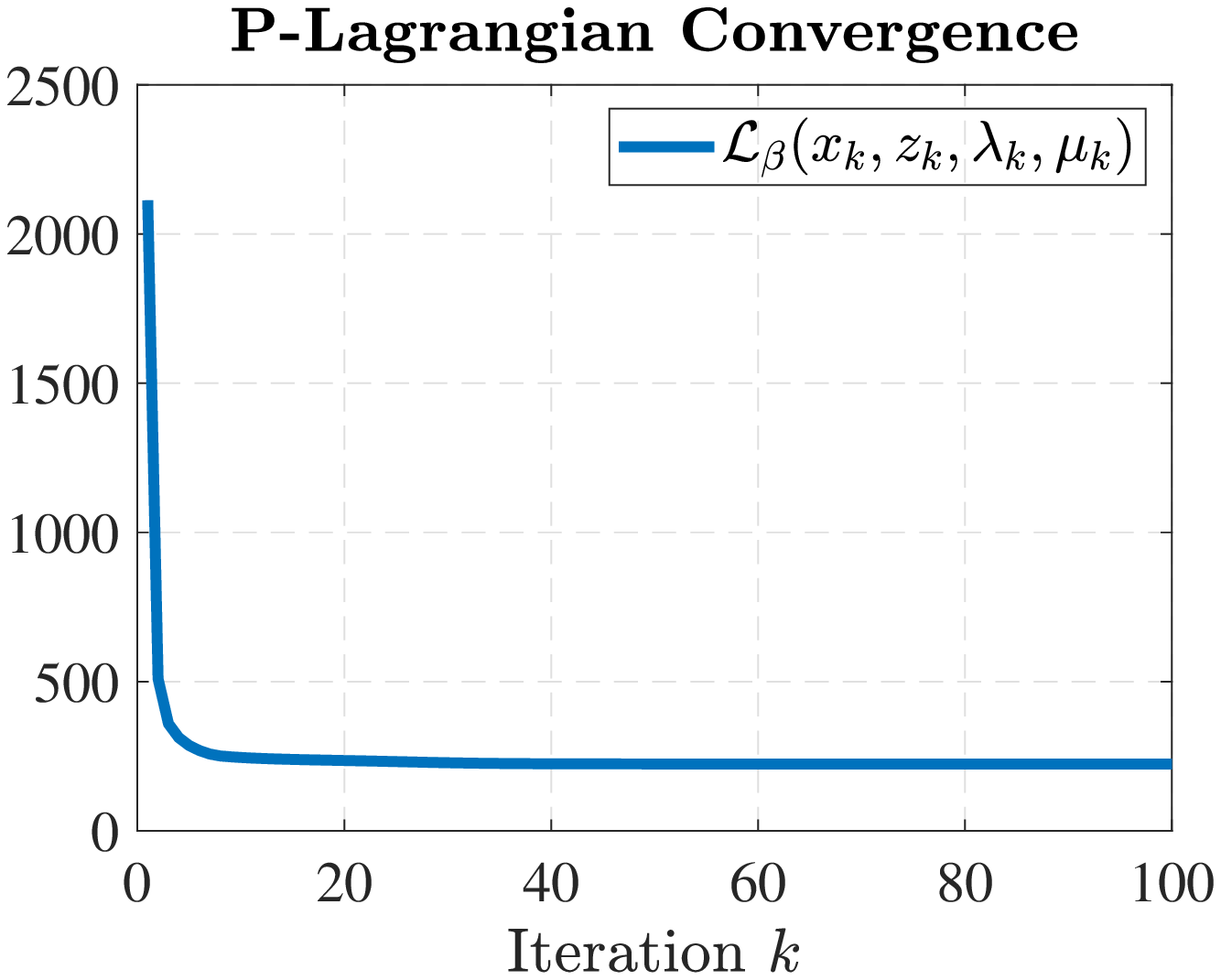}

\includegraphics[scale=0.38]{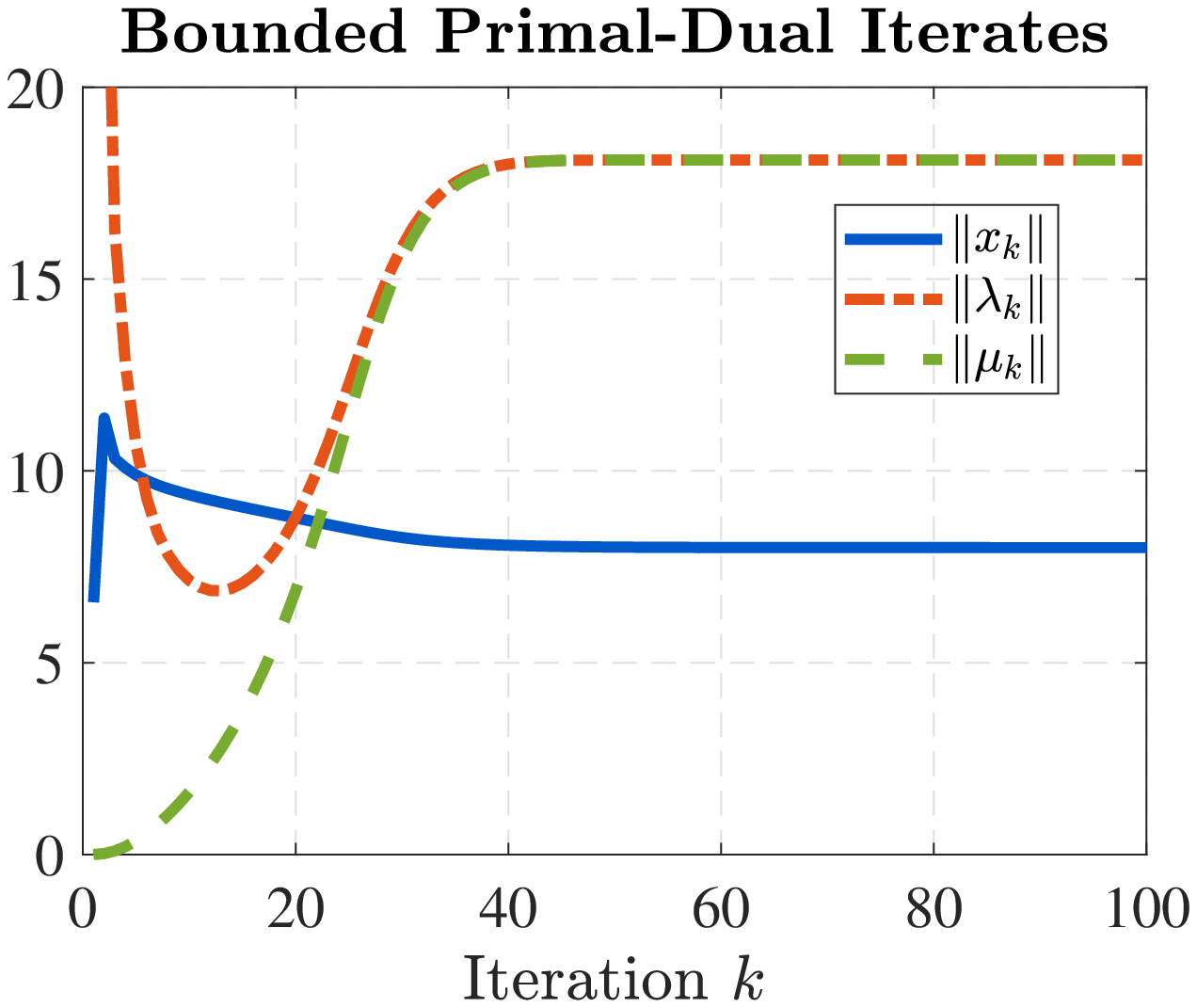}

\includegraphics[scale=0.38]{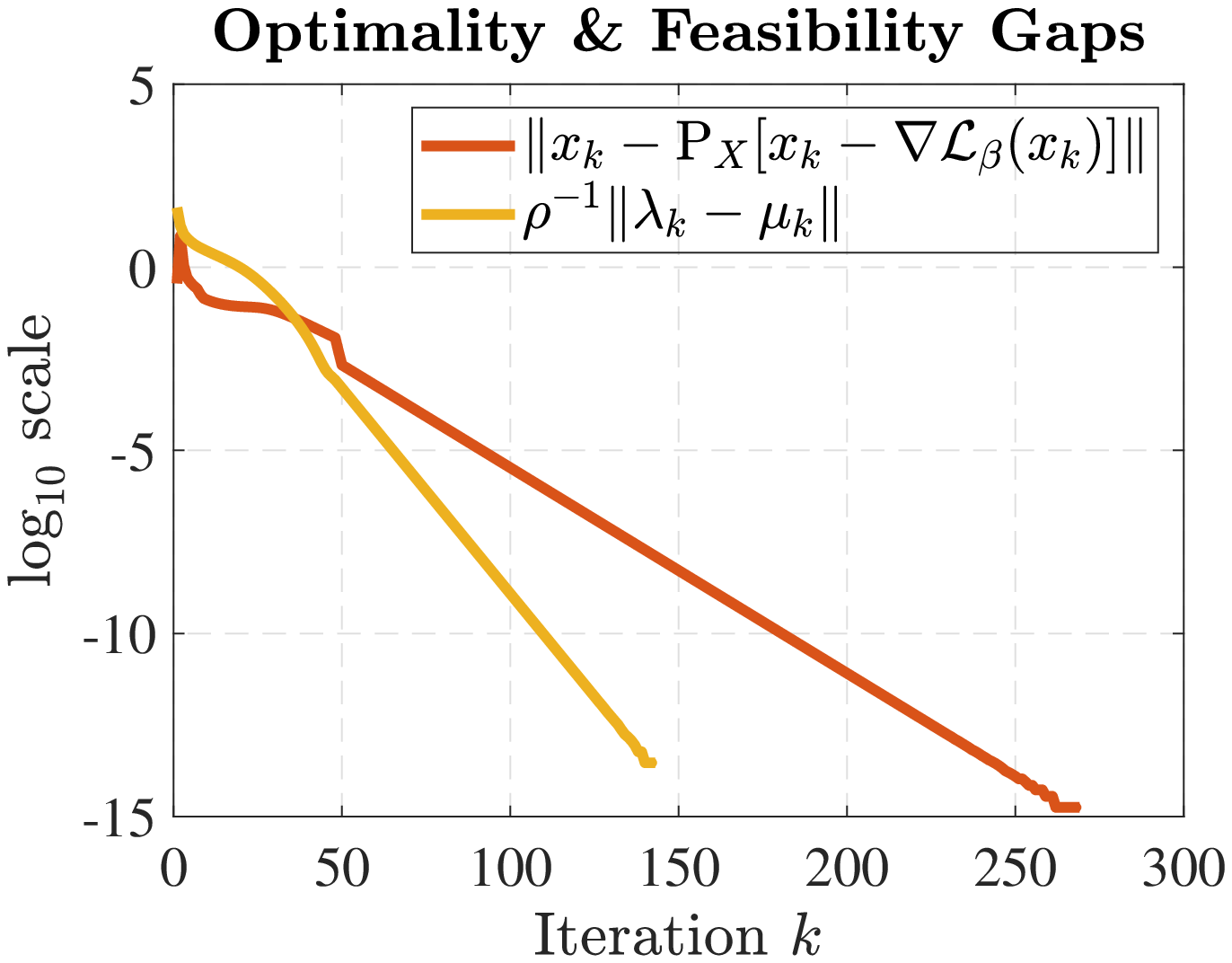}
\end{multicols}
 \caption{Convergence behavior of Algorithm \ref{algorithm1} applied to Example \ref{example2}. The initializations are $x_0=(4,4,4)$ and $(z_0,\lambda_0,\mu_0)=(0,0,0)$. The parameters are set to $\eta^{-1}=0.005, \alpha=2000$,  $\beta=0.5$, $\delta_0=0.5$, and $r=0.999$.} \label{fig:ex2}
 \end{figure*}

\begin{example}[Nonconvex functional constraint] \label{example2}
Consider the nonconvex quadratically constrained quadratic program (QCQP):
\begin{equation} 
\begin{aligned}
\underset{ x_1, x_2 \in \mathbb{R}_+  }{\mathrm{min}}  & \ \ f(x)  =\frac{1}{2}x^T Qx + q^Tx  \\
 \mathrm{s.t.}  & \ \ c_1(x) = \frac{1}{2}x^T Q_1x + q_1^Tx +128=0, \\
	          & \ \ c_2(x) =  \frac{1}{2}x^T Q_2x + q_2^Tx +32=0,
\end{aligned} \notag
\end{equation}
where
\begin{align}
& Q=\begin{bmatrix}
-2 & 10 & 2\\
10 & 4  & 1 \\
2  & 1 & -7
\end{bmatrix},  
 Q_1=\begin{bmatrix}
1 & 0  & 0 \\
0 & -1 & 0 \\
0 & 0  & 4
\end{bmatrix}, 
 Q_2=\begin{bmatrix}
1 & 0 & 0 \\
0 & 1  & 0 \\
0  & 0 & 1
\end{bmatrix}, \ 
 q = \begin{bmatrix}
-12 \\ -6 \\ 56
\end{bmatrix}, 
 q_1 = \begin{bmatrix}
0 \\ 0 \\ -32
\end{bmatrix}, 
 q_2 = \begin{bmatrix}
0 \\ 0 \\ -8
\end{bmatrix}. \notag
\end{align}
\end{example}
The optimal solution to Example \ref{example2} is $(x_1^\ast,x_2^\ast,x_3^\ast)=(0,0,8)$ and the optimal value is 224. Since $Q_1$ is indefinite, $c_1(x) $ is nonconvex, and the problem violates LICQ at $(0,0,8)$. Algorithm \ref{algorithm1} achieves the optimal solution. Figure \ref{fig:ex2} indeed shows the convergence of P-Lagrangian,  boundedness of dual iterates, and optimality and feasibility. 

We consider the following mathematical program with complementarity constraints (MPCC). 
\begin{example}[MPCC] \label{example3}
\begin{equation} 
\begin{aligned}
\underset{ x_1, x_2 }{\mathrm{min}} & \ \   x_1^2 + x_2^2 - 4x_1 x_2  \\
\mathrm{s.t.} & \ \ x_1^2 - x_2^2-4=0 \\
                & \ \ x_1 x_2 =0, \ x_1,x_2 \geq 0 \ (\Leftrightarrow  0 \leq x_1 \perp x_2  \geq 0). \notag
\end{aligned} 
\end{equation}
\end{example}
Observe that all constraints of  Example \ref{example3} are nonconvex, and LICQ is not satisfied. The numerical results are illustrated in Figure \ref{fig:ex3}, from which we see that Algorithm 1 converges to the optimal solution $x^\ast=(2,0)$.

\begin{figure*}[ht!] 
\centering
\begin{multicols}{3}
 \includegraphics[scale=0.38]{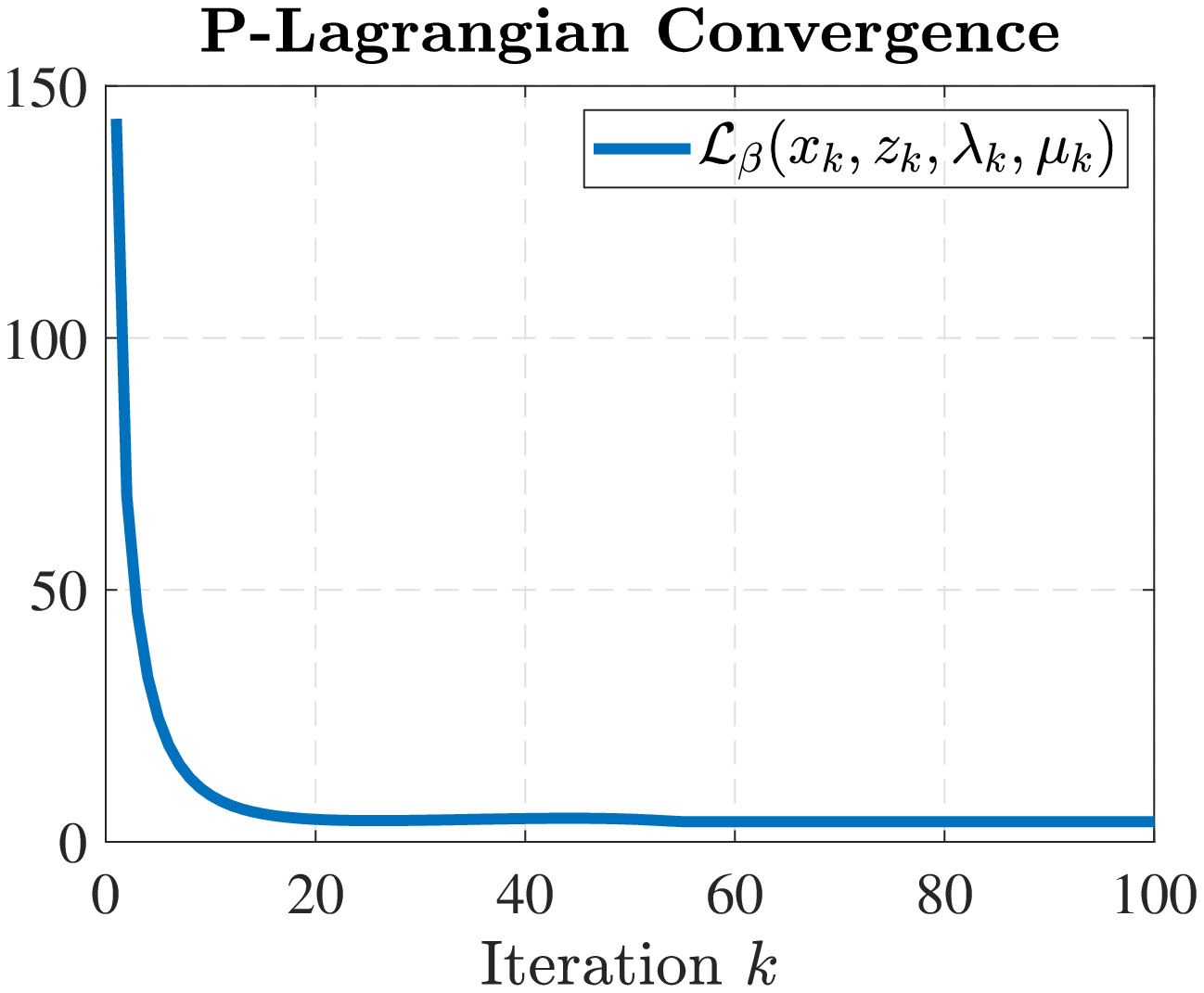}
 	
\includegraphics[scale=0.38]{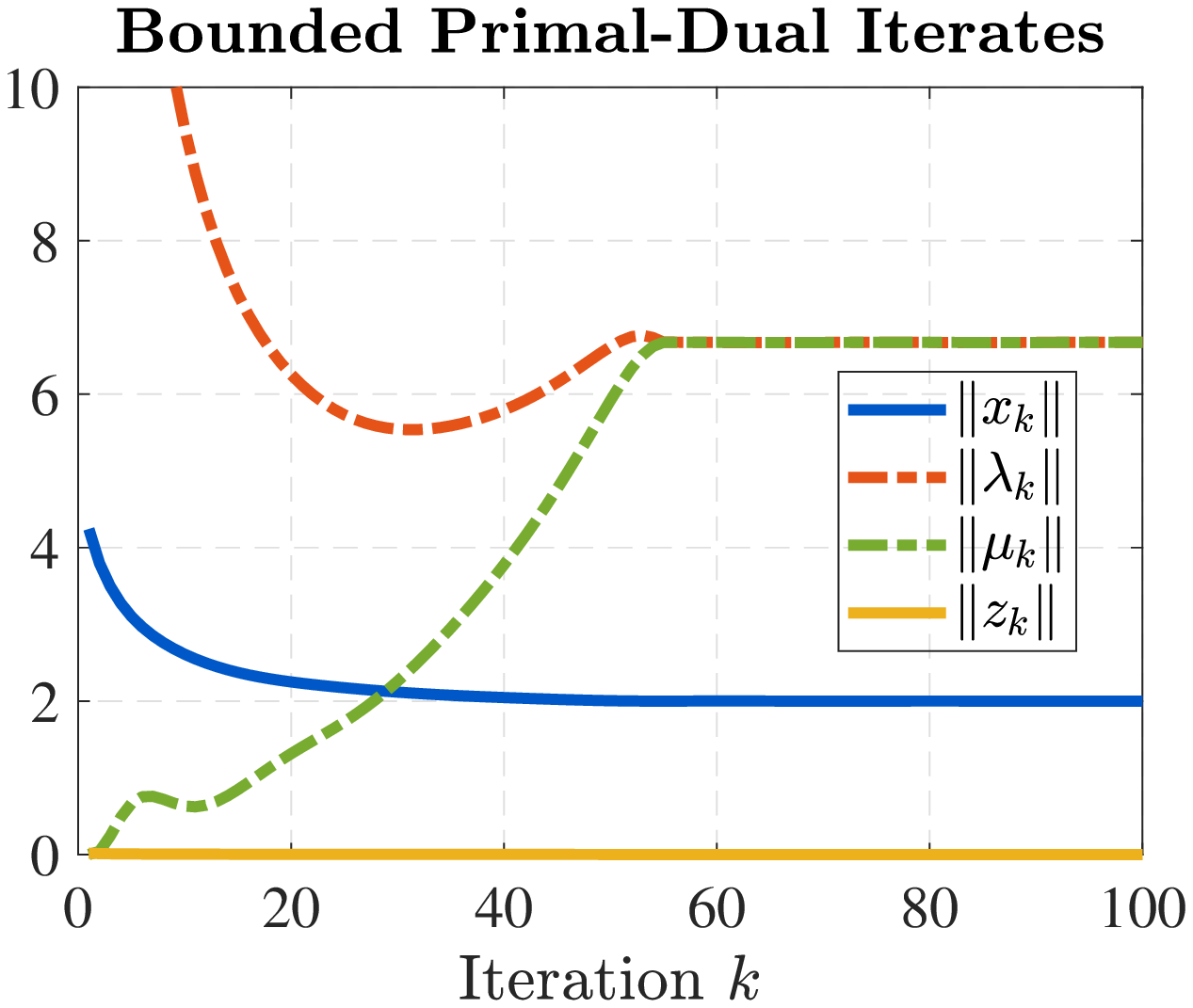} 

\includegraphics[scale=0.38]{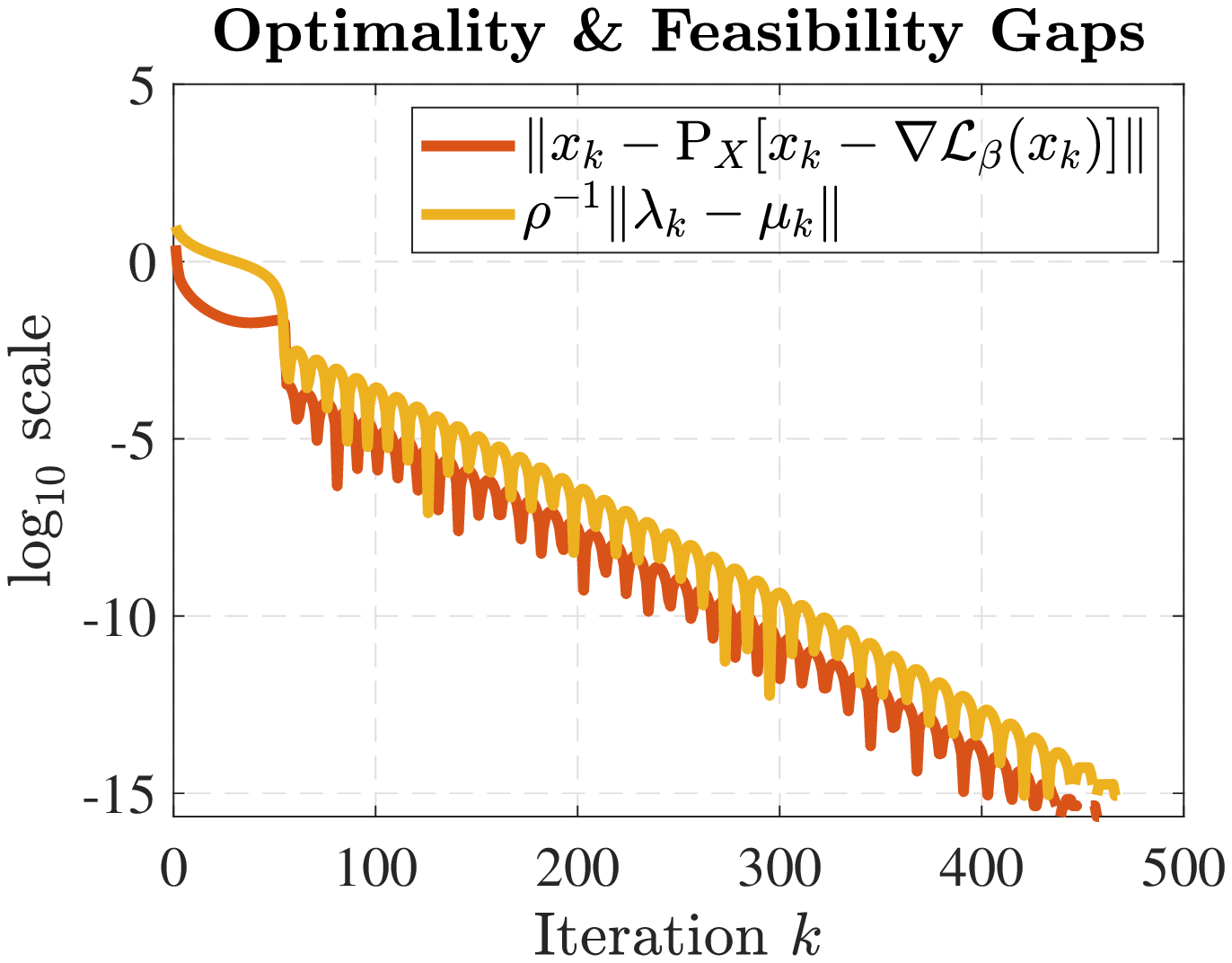}
\end{multicols}
 \caption{Convergence behavior of  Algorithm \ref{algorithm1} applied to Example \ref{example3}. The initialization is set to $x_0=(5,5)$ with $(z_0,\lambda_0,\mu_0)=(0,0,0)$. The parameters are set to $\eta^{-1}=0.004, \alpha=2000$,  $\beta=0.5$, and  $\delta_0=0.5$ with $r=0.999$.} \label{fig:ex3}
 \end{figure*}

\section{Conclusions} \label{sec:conclusion}
This paper studies the convergence of a new Lagrangian-based method for nonconvex optimization problems with nonlinear equality constraints. We presented a novel algorithmic framework based on the Proximal-Perturbed Lagrangian. We have shown our method improves existing AL-based algorithms; the method does not require boundedness assumptions on the iterates and set of multipliers, and it does not include the penalty terms for feasibility violation, which naturally leads to a simple single-loop algorithm. Possible future research is to extend this method to nonconvex-nonsmooth composite optimization settings,  which will result in a broader application domain.

\bibliographystyle{abbrv}
\bibliography{references.bib}

\end{document}